\documentclass[a4paper,10pt]{article}
\usepackage[utf8]{inputenc}
\usepackage[english]{babel}
\usepackage{amsmath,amsthm,amssymb}
\usepackage{url}
\usepackage{eucal}
\usepackage{esint}
\usepackage{bm}
\usepackage{fancyhdr}
\usepackage{color}
\usepackage{tikz}
\usepackage{graphicx}
\usepackage[font=small]{caption}

\newtheorem{theorem}{Theorem}
\newtheorem{proposition}[theorem]{Proposition}
\newtheorem{lemma}[theorem]{Lemma}
\newtheorem{corollary}[theorem]{Corollary}
\theoremstyle{definition}
\newtheorem{definition}[theorem]{Definition}
\newtheorem{remark}[theorem]{Remark}

\numberwithin{equation}{section}
\numberwithin{theorem}{section}

\newcommand{\ddiv}{\operatorname{div}}

\newcommand{\card}{\operatorname{card}}

\newcommand{\T}{\mathcal{T}}
\newcommand{\Cor}{\boldsymbol{\mathcal C}}
\newcommand{\nei}{\mathsf{N}}
\newcommand{\wcba}{\bm{\mathrm{wcba}}}
\newcommand{\Rd}{{\mathbb{R}^d}}
\newcommand{\Cov}{\operatorname{Cov}}

\begin{document}
\pagestyle{fancy}
\fancyhead{}
\setlength{\headheight}{14pt}
\renewcommand{\headrulewidth}{0pt}

\fancyhead[c]{\small \it Numerical stochastic homogenization}
\title{A priori error analysis of a\\numerical stochastic homogenization method}

\author{%
       Julian Fischer\thanks{%
         Institute of Science and Technology Austria (IST Austria),
Am Campus 1, 
3400 Klosterneuburg, 
Austria}
         \and
       Dietmar Gallistl\thanks{%
         Institut f\"ur Mathematik,
         Universit\"at Jena,
         07737 Jena, Germany}
        \and 
       Daniel Peterseim\thanks{%
            Institut f\"ur Mathematik,     
            Universit\"at Augsburg,
            86159 Augsburg, Germany}
       }

\date{}

\maketitle

\begin{abstract}
This paper provides an a~priori error analysis
of a localized orthogonal decomposition method (LOD) for the numerical stochastic homogenization of a model random diffusion problem.
If the uniformly elliptic and bounded random coefficient field of the model problem
is stationary and satisfies
a quantitative decorrelation assumption
in form of the spectral gap inequality, then the expected $L^2$ error of the method can be estimated, up to logarithmic factors,
by $H+(\varepsilon/H)^{d/2}$; $\varepsilon$ being the small correlation 
length of the random coefficient and $H$ the width of the coarse finite element 
mesh that determines the spatial resolution. The proof bridges recent results of numerical homogenization and quantitative stochastic homogenization.
\end{abstract}

{\small

}

\section{Introduction}

We study a prototypical random diffusion problem
\begin{equation*}
 -\ddiv \mathbf A \nabla \mathbf u = f
\end{equation*}
with homogeneous Dirichlet boundary conditions on a bounded Lipschitz
polytope.
The diffusion tensor $\mathbf A$ is a random coefficient field
with short correlation length $\varepsilon>0$.
We are interested in the approximation of this random partial differential equation by a deterministic finite 
element model and corresponding estimates of the expected $L^2$ error. The approximation is based on the localized orthogonal decomposition (LOD) approach to numerical homogenization beyond scale separation and periodicity \cite{MalqvistPeterseim2014,HenningPeterseim2013,Peterseim:2015}. This method is well established for deterministic applications ranging from  non-linear, time-dependent, multi-physics problems \cite{HMP13x,altmann2018computational,MP17,2017arXiv171009609V} to the stabilization of numerical wave scattering \cite{P17,GallistlPeterseim2015,PS17}. Apart from possible reinterpretations of the approach in the frameworks of domain decomposition \cite{Kornhuber.Peterseim.Yserentant:2016,peterseim2018domain} and Bayesian inference \cite{Owhadi2015,Owhadi2017}, the method can be rephrased as a discrete nonlocal integral operator
with a piecewise constant and exponentially decaying integral kernel thereby connecting the approach to the mathematical theory of homogenization \cite{GallistlPeterseim2017}. This particular perspective extends to the present stochastic homogenization problem in a natural way. In \cite{GallistlPeterseim2019} it is shown that the expectation of the discrete nonlocal integral representations of the realizations of the random operator provides an approximation of the stochastically homogenized operator. The error bounds of this LOD approach to stochastic homogenization contains the typical a~priori terms for the spatial discretization, quantified 
by the mesh size (or observation scale) $H$, and an a~posteriori
estimator that represents local fluctuations of the deterministic
upscaled model. Without any assumptions on the statistical structure of the coefficient $\mathbf{A}$, the a~priori quantification of the statistical error estimator seems hardly possible. However, the numerical experiments of \cite{GallistlPeterseim2019} revealed that
small values of 
the estimator are achieved given a certain scale separation in the 
stochastic variable, in particular for random coefficients at finite correlation length $\varepsilon$. This paper makes this plausible observation rigorous in an a~priori error analysis
that is explicit in
$H$ and $\varepsilon$. 

The key tools for our present paper are adopted from the recent quantitative theory of homogenization, in particular the framework of functional inequalities from \cite{GloriaOtto2011,GloriaOtto2012,
GloriaNeukammOtto2015,GloriaNeukammOttoPreprint} and the regularity theory from \cite{ArmstrongSmart2016,ArmstrongMourrat2016,
GloriaNeukammOttoPreprint,ArmstrongDaniel,FischerRaithel,DuerinckxOtto2019}.
We recall that the quantitative theory of stochastic homogenization has recently lead to optimal-order convergence rates for linear elliptic PDEs with random coefficient field \cite{GloriaOtto2011,GloriaOtto2012,ArmstrongKuusiMourrat2017,
GloriaOttoNew}, as well as to a corresponding result for monotone operators \cite{FischerNeukamm}.
For (non-optimal) convergence rates for further nonlinear problems, we refer to \cite{ArmstrongCardaliaguet,ArmstrongCardaliaguetSouganidis,
ArmstrongFergusonKuusi2}.
An overview of computational methods in stochastic homogenization
can be found in the review article
\cite{AnantharamanEtAl_Review2012}; see also 
\cite{ArmstrongKuusiNumerical,Gu2019Preprint}.
Numerical approaches to the computation of effective coefficients in stochastic homogenization have been devised e.\,g.\ in \cite{BourgeatPiatnitski2004,CancesEhrlacherEtAl,KhoromskaiaKhoromskijOtto,
MourratNumerical};
of particular interest in this context are variance reduction schemes, see \cite{LeBrisLegollMinvielle,BlancLeBrisLegoll} for several methods capable of substantially reducing the computational cost and \cite{FischerVarianceReduction} for a theoretical analysis.

Altogether, by merging the theories of LOD and quantitative stochastic homogenization we achieve rigorous a priori error bounds for a numerical stochastic homogenization in the spirit of LOD. If the uniformly elliptic and bounded random coefficient field $\mathbf A$ is 
stationary and satisfies a quantitative decorrelation assumption
in form of the spectral gap inequality then the numerical deterministic approximation $u_H$ of the random solution field $\mathbf u$ fulfills the relative error bound
\begin{equation}\label{e:mainintro}
 \frac{\mathbb{E}\bigl[ \| \mathbf{u}-u_H \|_{
			L^2(D)}^{2}\bigr]^{1/2}}{\|f\|_{L^2(D)}}
\lesssim
\lvert\log H\rvert^{4+3d/2}
\left(H + \left(\frac{\varepsilon}{H}\right)^{d/2}\right).
\end{equation}
Recall that $H>0$ refers to the mesh size of some possibly coarse simplicial finite element mesh $\mathcal{T}_H$ that underlies the LOD construction and $\varepsilon$ is the correlation length of $\mathbf{A}$. The estimate \eqref{e:mainintro} appears to be optimal in the sense of spatial approximability and CLT scaling (up to the logarithmic factor which is most probably pessimistic). This bound is in agreement with the numerical experiments of \cite{GallistlPeterseim2019} for a relevant class of random coefficients in the regime of short-range correlation. We shall emphasize that the method of \cite{GallistlPeterseim2019} itself is applicable without the structural assumptions of stationarity and quantitative decorrelation. However, the accuracy of a deterministic approximation of the random solution field is very limited beyond such assumptions. 

Apart from the mathematical justification of LOD for stochastic homogenization problems, the numerical analysis of this paper may have impact on the practical realization of more general multiscale representations of homogenized operators in stochastic homogenization \cite{FP18}. Moreover, given the aforementioned generalizations of both LOD and the analytical techniques, the present work may be the starting point for the numerical analysis of more involved stochastic homogenization problems beyond the prototypical linear elliptic model problem.

The remaining parts of this paper are structured
 as follows.
Section~\ref{s:model} specifies the model problem and Section~\ref{s:notation} reviews the numerical stochastic homogenization method of
\cite{GallistlPeterseim2019}. Section~\ref{s:erroranalysis} characterizes the admissible class of random diffusion coefficients and presents and proves the main results.

\bigskip
Standard notation on Lebesgue and Sobolev spaces
applies throughout this paper. 
The notation $a\lesssim b$ abbreviates $a\leq C b$ for some
constant $C$ that is independent of the mesh size and variations of the coefficient $\mathbf{A}$ but may depend on the shape of mesh elements and the contrast (i.e., the ratio of the uniform upper and lower bound) of $\mathbf{A}$ (cf. Assumption (A1) of Subsection~\ref{ss:keyassumptions}). The notation 
$a\approx b$ abbreviates $a\lesssim b\lesssim a$. For a matrix $\mathbf{A}$, we use the notation $\mathbf{A}^*$ to denote its transpose.
The duality product of $H^{-1}$ and $H^1_0$ is denoted by
$\langle\cdot,\cdot\rangle$.

\section{Model problem}\label{s:model}
Let $(\Omega,\mathcal{F},\mathbb{P})$ be a probability space with set of events $\Omega$, $\sigma$-algebra $\mathcal{F}\subseteq 2^\Omega$, and probability measure $\mathbb{P}$. 
The expectation operator is denoted by $\mathbb{E}$.
Let $D\subseteq\mathbb R^d$ for $d\in\{1,2,3\}$ be
a bounded Lipschitz polytope with a diameter of order $1$. 
For technical reasons in our proofs we assume that 
$D$ is a cuboid.
Let $\mathbf{A}$ be a uniformly elliptic and bounded random coefficient field
and let, for the sake of readability, 
$\mathbf{A}$ be pointwise symmetric.
The proofs can, however, be extended to the unsymmetric case.
For a deterministic right-hand side $f\in L^2(D)$
the model problem reads
\begin{equation}\label{e:modelstrong}
\left\{
\begin{aligned}
-\ddiv (\mathbf{A}(\omega)(x) \nabla \mathbf{u}(\omega)(x) ) &= f(x),&x\in D\\
\mathbf{u}(\omega)(x)&=0,&x\in \partial D
\end{aligned}
\right\}\quad\text{for almost all }\omega\in\Omega.
\end{equation}
The weak formulation of \eqref{e:modelstrong} is based on the 
Sobolev space $V:=H^1_0(D)$ 
and seeks a random field
$\mathbf{u}$ in the Hilbert space 
$L^2(\Omega;V)$
such that 
\begin{equation}\label{e:diff1drandweak2}
\int_\Omega\int_D(\mathbf{A}(\omega)\nabla\mathbf{u}(\omega)(x))\cdot\nabla\mathbf{v}(\omega)(x)\,dx\,d\mathbb{P}(\omega) = \int_\Omega\int_D f(x) \mathbf{v}(\omega)(x)\,dx\,d\mathbb{P}(\omega)
\end{equation}
holds for all $\mathbf{v}\in L^2(\Omega;V)$.
Well-posedness of this problem follows from coercivity of the bilinear form on the left-hand side. 

The numerical stochastic homogenization method introduced below can be applied to the model problem without further statistical assumptions on the diffusion coefficient $\mathbf A$. However, its a~priori error analysis based on the quantitative theory of stochastic homogenization will require the restriction to the class of stationary random coefficient fields $\mathbf A$ satisfying a spectral gap inequality. These structural assumptions will be made specific in Subsection~\ref{ss:keyassumptions}.

\section{Numerical stochastic homogenization method}\label{s:notation}
This section reviews the numerical stochastic homogenization method of \cite{GallistlPeterseim2019}. This requires the introduction of some basic notation on finite element spaces. 
\subsection{Finite element notation}
Let $\T_H$ denote a quasi-uniform and regular simplicial triangulation of 
the domain $D$. Introduce the global mesh size  
$H:=\max\{\operatorname{diam}(T):T\in\T_H\}$ of the quasi-uniform mesh $\T_H$. The corresponding $P_1$ finite element space of piecewise affine
and globally continuous functions that satisfy the homogeneous Dirichlet boundary condition
is denoted by $V_H\subseteq V$. This space will be used for the approximation of the solution. That is why we refer also to the corresponding discretization scale $H$ as the observation scale. For the approximation of integral kernels that define the numerical homogenization method we will use the space of piecewise constant functions
(resp.\ $d\times d$ matrix fields) which is denoted by 
$P_0(\T_H)$ (resp.\ $P_0(\T_H;\mathbb R^{d\times d})$).

The definition of localized numerical correctors requires the concept of patches. 
The neighborhood (or first-order patch) of a
given subdomain $S\subseteq\overline D$ is defined as
\begin{equation*}
\nei(S):=\operatorname{int}
          \Big(\bigcup\{T\in\T_H\,:\,T\cap\overline S\neq\emptyset \}\Big).
\end{equation*}
Furthermore, we introduce for any $\ell\geq 2$ the patch extensions
\begin{equation*}
\nei^1(S):=\nei(S)
\qquad\text{and}\qquad
\nei^\ell(S):=\nei(\nei^{\ell-1}(S)) .
\end{equation*}
Note that the number of elements in the $\ell$th-order patch in a quasi-uniform mesh scales like $\ell^d$. 

Let $I_H:V\to V_H$ be a 
surjective
quasi-interpolation operator that
acts as an $H^1$-stable and $L^2$-stable
quasi-local projection in the sense that
$I_H\circ I_H = I_H$ and that
for any $T\in\T_H$ and all $v\in V$ there holds
\begin{align}
\label{e:IHapproxstab}
H^{-1}\|v-I_H v\|_{L^2(T)} + \|\nabla I_H v \|_{L^2(T)}
&
\leq C_{I_H} \|\nabla v\|_{L^2(\nei(T))} 
\\
\label{e:IHstabL2}
\|I_H v\|_{L^2(T)}
&
\leq C_{I_H} \|v\|_{L^2(\nei(T))} .
\end{align}
For the discussion in this paper, we choose $I_H$ to be 
the concatenation of the $L^2$ projection to (possibly discontinuous)
piecewise affine functions over $\T_H$
and the averaging operator that maps a piecewise affine
function $p_H$ to $V_H$ by assigning to each interior
vertex $z$ the average of all values $p_H|_T(z)$ 
such that $z\in T\in\T_H$
\cite{GallistlPeterseim2019}.
We note that various other choices are possible.

\begin{remark}\label{r:submesh}
Let $\tilde \T_H$ be the triangulation generated from
$\T_H$ by two barycentric refinements
and let $\tilde V_H\subseteq V$ denote the first-order
finite element space with respect to $\tilde \T_H$.
For the above choice of $I_H$, it is easy to see that
for any $v\in V$ there exists $\tilde v\in\tilde V_H$
such that $I_H v = I_H \tilde v$
and $\|\tilde v\|_{L^2(\Omega)}\lesssim \|v\|_{L^2(\Omega)}$.
This remains true if $I_H$ is replaced by the operator
on $H^1(\Omega)$ that averages at boundary
vertices as well (i.e., not enforcing Dirichlet conditions).
The claim follows from the fact that any piecewise affine $p_H$
can be generated by the $L^2$ projection of a suitable
linear combination of bubble functions from $\tilde V_H$.
\end{remark}

\subsection{Numerical stochastic homogenization method}\label{ss:upscalingQloc}

The LOD approach of \cite{GallistlPeterseim2019} to stochastic homogenization
computes a quasi-local effective
coefficient as a discrete integral operator on finite element
spaces. Its construction is described in the following steps.

Following \cite{MalqvistPeterseim2014}, coarse and fine scales are characterized
through the quasi-interpolation operator $I_H$ introduced above. The 
space $W$ of fine-scale functions is defined by 
$W:=\operatorname{ker}I_H\subseteq V$.
Given a nonnegative integer 
\emph{oversampling parameter} $\ell$,
which throughout this paper is assumed to satisfy
$\ell\approx\lvert\log H\rvert$,
consider the 
$\ell$-th order extended patch $D_T:=\nei^\ell(T)$
of an element $T\in\T_H$.
The space of fine-scale functions that vanish outside $D_T$
is denoted by 
$W_{D_T}\subseteq W$. Note that this choice encodes a homogeneous Dirichlet boundary condition at the boundary of $D_T$.

Given the $j$-th Cartesian unit vector $e_j$ ($j=1,\dots,d$),
the localized element corrector 
$\mathbf{q}_{T,j}\in L^2(\Omega;W_{D_T})$ 
related to the element $T\in\T_H$ is defined as the solution
to the following localized problem (cell problem)
\begin{equation}\label{e:qTelementcorrELL}
\int_{D_T} \nabla w\cdot(\mathbf{A} \nabla \mathbf{q}_{T,j})\,dx
=
\int_{T} \nabla w \cdot(\mathbf{A} e_j)\,dx
\quad\text{for all }w\in W_{D_T}.
\end{equation}
Given $v_H\in V_H$, we define the correction operator 
$\Cor v_H\in L^2(\Omega;W)$
by
\begin{equation}\label{e:corexpansionELL}
\Cor  v_H
= \sum_{T\in\T_H} \sum_{j=1}^d (\partial_j v_H|_T)  \mathbf{q}_{T,j}.
\end{equation}
Note that the element correctors and the correction operator implicitly depend 
on the parameter $\ell$.
In the deterministic case it was shown in \cite{GallistlPeterseim2017}
how the use of corrected test functions 
leads to a sparse discrete integral operator.
In the stochastic setting, a similar representation with a
stochastic integral kernel is possible
\cite{GallistlPeterseim2019}, namely
with the piecewise-in-space constant matrix field
$\boldsymbol{\mathcal A}_H\in L^2(\Omega;P_0(\T_H\times\T_H;\mathbb{R}^{d\times d}))$
over $\T_H\times\T_H$,
which, for $T,K\in \T_H$, is defined by
\begin{equation}\label{e:AHdef}
(\boldsymbol{\mathcal A}_H|_{T,K})_{jk}
:= 
\frac{1}{|T|\,|K|}
\left(\delta_{T,K}\int_T \mathbf{A}_{jk}\,dx -
 e_j\cdot\int_K \mathbf{A}\nabla \mathbf{q}_{T,k}\,dx 
\right)
\end{equation}
($j,k=1,\dots,d$)
where $\delta$ is the Kronecker symbol. Note that the operator $\boldsymbol{\mathcal A}_H$ is sparse in the sense that
$\boldsymbol{\mathcal A}_H|_{T,K}$ equals zero for $T,K\in\T_H$ whenever $K\notin \nei^\ell(T)$, i.e., 
$\operatorname{dist}(T,K)\gtrsim \ell H$.

The kernel $\boldsymbol{\mathcal A}_H$ induces the discrete bilinear form 
$\boldsymbol{\mathfrak a}: V_H\times V_H\to L^2(\Omega;\mathbb R)$ given by
$$
\boldsymbol{\mathfrak a}(v_H,z_H):=
\int_D\int_D 
   \nabla v_H(x) \cdot ( \boldsymbol{\mathcal A}_H(x,y)\nabla z_H(y))\,dy\,dx 
$$
for any  $v_H,z_H\in V_H$. $V$-coercivity and continuity of the form $\boldsymbol{\mathfrak a}$
for any $\omega\in\Omega$
under the condition
$\ell\approx\mathcal{O}(\lvert\log H\rvert)$ were shown
in \cite{GallistlPeterseim2017}.

As pointed out in \cite{GallistlPeterseim2017}, 
there holds for all finite element functions $v_H,z_H\in V_H$ that
\begin{equation}\label{e:frakaeq}
\int_D \nabla v_H\cdot (\mathbf{A}\nabla(1-\Cor) z_H)\,dx
=
\mathfrak{a}(v_H,z_H).
\end{equation}
This shows how the form $\boldsymbol{\mathfrak a}$ is connected to a Petrov-Galerkin
variant of the method of \cite{MalqvistPeterseim2014,GallistlPeterseim2017}.

The final approximation by a deterministic model is based on the 
averaged integral kernel
$\bar{\mathcal A}_H :=\mathbb{E}[ \boldsymbol{\mathcal A}_H]$,
i.e.,
\begin{equation}\label{e:barAHdef}
(\bar{\mathcal A}_H|_{T,K})_{jk}
= 
\frac{1}{|T|\,|K|}
\left(\delta_{T,K}\int_T \mathbb{E}[\mathbf{A}_{jk}]\,dx -
 e_j\cdot\int_K \mathbb{E}[\mathbf{A}\nabla \mathbf{q}_{T,k}]\,dx  
\right) 
\end{equation}
for any two simplices $T,K\in\T_H$.
The corresponding deterministic bilinear form
$\bar{\mathfrak{a}}(\cdot,\cdot)$
is given by
$$
\bar{\mathfrak{a}}(v_H,z_H):=
\int_D\int_D 
   \nabla v_H(x) \cdot (\bar{ \mathcal A}_H(x,y)\nabla z_H(y))\,dy\,dx 
   \quad\text{for any } v_H,z_H\in V_H.
$$ 
Given this discrete deterministic approximation of the random partial differential operator, an approximation $u_H\in V_H$ of the solution $\mathbf u$ in the coarse finite element space $V_H$ solves
\begin{equation}\label{e:uHdef}
\bar{\mathfrak{a}}(u_H,v_H)
= (f,v_H)_{L^2(D)}
\quad\text{for all }v_H\in V_H.
\end{equation}
This can be phrased as a sparse linear system using the canonical nodal basis of the finite element space. Compared to the direct finite element discretization of \eqref{e:diff1drandweak2}, this system is slightly denser because the degrees of freedom associated with the interior vertices of $\T_H$ are directly coupled over distances of order $\ell H$. In this regard the system is similar to a $B$-spline discretization of order $\ell+1$.

Altogether, the numerical stochastic homogenization method consists of two steps. The first step is the assembling of the system matrix associated with \eqref{e:uHdef} which in turn requires the solution of the $d\times\card\T_H$ cell problems \eqref{e:qTelementcorrELL}. This task is often referred to as the offline phase which is independent of the right-hand side. In analogy to periodic deterministic coefficients \cite{GallistlPeterseim2015}, stationarity plus an appropriately chosen structured mesh $\T_H$ allow one to reduce the number of cell problems to $\mathcal{O}(\ell^d)$ (namely
$\mathcal{O}(1)$ representative interior problems plus 
all representative intersections of patches with the domain boundaries). 
We shall emphasize in this connection that, in contrast to analytical approaches to homogenization, the numerical method depends on the domain $D$ through the mesh and the boundary condition encoded in the cell problems. This dependence can be eliminated by replacing $\bar{\mathcal A}_H$ with a Toeplitz matrix resulting from solving a representative interior cell problem. In this spirit, one may as well approximate $\bar{\mathcal A}_H$ by a diagonal matrix resulting from row averaging to recover a classical finite element system that can be interpreted as the discretization of the homogenized PDE in certain cases \cite{GallistlPeterseim2017}. However simplification steps are beyond rigorous a~priori error control and will not be discussed further here.  

\section{Error analysis}\label{s:erroranalysis}
This section presents the novel a priori error analysis that combines arguments from the theories of LOD and quantitative stochastic homogenization. This requires some structural assumptions on the underlying random diffusion field.

\subsection{Key assumptions}\label{ss:keyassumptions}
We will impose three structural assumptions on the random coefficient field $\mathbf{A}$: uniform ellipticity and boundedness, stationarity, and quantitative decorrelation. These conditions are classical in stochastic homogenization, see for instance \cite{GloriaNeukammOttoPreprint}.  
\begin{itemize}
\item[(A1)] The random coefficient field $\mathbf{A}$ is uniformly elliptic and bounded, i.e., there exist constants $0<\lambda\leq \Lambda<\infty$ such that almost surely we have $\mathbf{A}(x)v\cdot v\geq \lambda |v|^2$ and $|\mathbf{A}(x)v| \leq \Lambda |v|$ for every $v\in \Rd$ and almost every $x\in \Rd$.
\item[(A2)] The random coefficient field $\mathbf{A}$ is \emph{stationary}, that is the law of shifted coefficient field $\mathbf{A}(\omega)(\cdot+x)$ coincides with the law of $\mathbf{A}$ for all $x\in \Rd$.
\item[(A3)] The random coefficient field $\mathbf{A}$ is subject to a quantitative decorrelation assumption on scales larger than $\varepsilon$ in form of the spectral gap inequality with correlation length $\varepsilon>0$, i.e., there exists a constant $\rho>0$ such that for any Fr\'echet differentiable random variable $F=F(\mathbf{A})$ the estimate
\begin{align}
\label{SpectralGap}
\mathbb{E}\big[\big|F-\mathbb{E}[F]\big|^2\big]
\leq \frac{\varepsilon^d}{\rho} \mathbb{E}\Bigg[\int_\Rd \bigg|\fint_{B_\varepsilon(x)} \bigg|\frac{\partial F}{\partial \mathbf{A}}(\tilde x)\bigg| \,d\tilde x\bigg|^2 \,dx\Bigg]
\end{align}
holds.
\end{itemize}

One example of a coefficient field satisfying the assumptions (A1)--(A3) are coefficient fields arising by applying a nonlinear function to a stationary Gaussian random field with integrable correlations. To be more explicit, let $k\in \mathbb{N}$ and let $Y:\Omega \times \Rd\rightarrow \mathbb{R}^k$ be a stationary Gaussian random field with integrable correlations in the sense
\begin{align*}
\int_\Rd \sup_{|\tilde x|=|x|} \big|\Cov[Y(\tilde x),Y(0)]\,\big| \,dx \lesssim \varepsilon^d.
\end{align*}
Furthermore, let $\xi:\mathbb{R}^k\rightarrow \mathbb{R}^{d\times d}$ be a $1$-Lipschitz function taking values in the space of matrices subject to the uniform ellipticity and boundedness conditions in (A1). Then the random field
\begin{align*}
\mathbf{A}(\omega)(x) := \xi(Y(\omega,x))
\end{align*}
satisfies the conditions (A1)--(A3) for some constant $\rho\gtrsim 1$.

\begin{remark}
For simplicity, we assume uniform ellipticity and boundedness of the
 coefficient field (condition (A1)). Since the ellipticity ratio of 
log-Gaussian random fields satisfies strong stochastic moment bounds 
we believe that similar results may be deduced in the case of 
log-Gaussian random fields or, more generally, random fields with moment
 bounds on
$\mathbb E[|A|^p]+\mathbb E[|A^{-1}|^p]$ 
for $p\gg1$ by using an adaptation of our strategy.

\end{remark}

\subsection{Review of a~posteriori error bounds}
We briefly review the $L^2$ error estimates from 
\cite{GallistlPeterseim2019} for the 
numerical method of the previous section which mark the starting point for the novel a~priori error analysis. These estimates require solely Assumption (A1) and no statistical assumptions. This generality results in an error estimate that  contains an a~posteriori term reflecting statistical errors that cannot be quantified a~priori without further assumptions. 

The error measure of interest is the 
$L^2(\Omega;L^2(D))$ norm. Besides the usual explicit convergence rates in terms of the 
mesh size $H$, the error bound contains an a~priori quantity
called worst-case best-approximation error, defined by
\begin{equation}\label{e:wcbadef}
\wcba(\mathbf{A}(\omega),\T_H) 
:=
\sup_{g\in L^2(D)\setminus\{0\}}
\inf_{v_H\in V_H} \frac{\|u(g,\mathbf{A}(\omega))- v_H\|_{L^2(D)}}{\|g\|_{L^2(D)}} \lesssim H
\end{equation}
where for $g\in L^2(D)$, $u(g,\mathbf{A}(\omega))\in V$ solves 
the deterministic 
model problem with diffusion coefficient $\mathbf{A}(\omega)$
and right-hand side $g$.
This quantity is always controlled from above by $H$, but it can
behave better (up to $H^2$) in certain regimes
\cite{GallistlPeterseim2017}.

The a~posteriori part in the error bound is referred to as
model error estimator.

\begin{definition}[model error estimator]\label{d:modelest}
	For any $T\in\T_H$, denote
	\begin{equation*}X(T) :=
	\max_{\substack{K\in\T_H\\ K\cap\nei^\ell(T)\neq\emptyset}}
	\lvert T\rvert 
	\, 
	\Big\lvert
	\boldsymbol{\mathcal A}_H|_{T,K}-\bar{ \mathcal A}_H|_{T,K}
	\Big\rvert  .
\end{equation*}
The model error estimator $\gamma$ is defined by
	\begin{equation*}
	\gamma:=\max_{T\in\T_H}\sqrt{\mathbb{E}[X(T)^2]}.
	\end{equation*}
\end{definition}
The model error estimator $\gamma$ coincides with the one introduced in \cite{GallistlPeterseim2019} up to some scaling factor that was used to improve the efficiency of the estimator in computations. Since the rescaling has no effect on the a priori error analysis it is not considered here. 
The following error estimate was shown in \cite[Prop. 9]{GallistlPeterseim2019}.
\begin{proposition}[error estimate for the quasilocal method]
	\label{p:errorestimateQlocal}
	Let $\ell\approx \lvert\log H\rvert$.
	Let $\mathbf{u}$ solve \eqref{e:diff1drandweak2} and let
	$u_H$ solve \eqref{e:uHdef} with right-hand side
	$f\in L^2(D)$.
	Then the estimate
	\begin{equation}\label{e:errEstQuasiEq1}
	\begin{aligned}
\sqrt{\mathbb{E}[ \| \mathbf{u}-u_H \|_{L^2(D)}^2] } 
	&
	\lesssim 
	(
	H^2 + \mathbb{E}[\wcba(\mathbf{A},\T_H)]
	+
	\ell^d  \gamma )
	\|f\|_{L^2(D)} 
	\\
	&
	\lesssim 
	(H + \ell^d  \gamma )  \|f\|_{L^2(D)}
	\end{aligned}
	\end{equation}
	holds with the model error estimator $\gamma$ from
	Definition~\ref{d:modelest}.
\end{proposition}

We end this paragraph by noting a technical perturbation
result that will later be used in the proof of 
Theorem~\ref{t:varestimate}.
\begin{lemma}\label{l:perturb}
 Let $T,K\in\T_H$ and $j,k\in\{1,\dots,d\}$.
Then there exist a box $Q\subseteq D_T$
and some $m\lesssim\ell$ such that
the patches satisfy the inclusion
$\nei^\ell(T) \subseteq \overline{Q} \subseteq \overline{\nei^m(T)}.$
Let $\mathbf{q}_{T,j}^Q\in W_Q$ solve \eqref{e:qTelementcorrELL}
with $D_T$ replaced by $Q$
and let 
$(\boldsymbol{\mathcal A}^Q_H|_{T,K})_{jk}$
be defined by \eqref{e:AHdef} with $\mathbf{q}_{T,j}$
replaced by $\mathbf{q}_{T,j}^Q$.
Then, the following perturbation result holds
almost surely
\begin{equation*}
 |(\boldsymbol{\mathcal A}_H|_{T,K})_{jk}
 -
 (\boldsymbol{\mathcal A}^Q_H|_{T,K})_{jk}|
 \lesssim
 \frac{H}{|T|}.
\end{equation*}
\end{lemma}

\begin{proof}
Denote $\widehat D_T:=\nei^m(T)$ and as before $D_T=\nei^\ell(T)$.
 The claimed inclusion relation 
follows from the quasi-uniformity of $\T_H$
and the assumption that the domain $D$ is rectangular.
From the definition \eqref{e:AHdef} applied to $D_T$ and $Q$
and the Hölder inequality
it follows that
\begin{equation*}
 |(\boldsymbol{\mathcal A}_H|_{T,K})_{jk}
 -
 (\boldsymbol{\mathcal A}^Q_H|_{T,K})_{jk}|
\lesssim
 |T|^{-1}|K|^{-1/2} 
 \|\nabla(\mathbf{q}_{T,j}-\mathbf{q}_{T,j}^Q)\|_{L^2(D_T)}.
\end{equation*}
By extending 
$\mathbf{q}_{T,j}$ and $\mathbf{q}_{T,j}^Q$ by zero to
functions from $W_{\widehat{D}_T}$
and noting that both functions are Galerkin projections
of the corrector
$\widehat{\mathbf{q}}_{T,j}\in W_{\widehat{D}_T}$ defined through
\eqref{e:qTelementcorrELL} with respect to
$\widehat{D}_T$, we deduce
\begin{equation*}
 \|\nabla(\mathbf{q}_{T,j}-\mathbf{q}_{T,j}^Q)\|_{L^2(D_T)}
 \lesssim
 \|\nabla(\mathbf{q}_{T,j}-\widehat{\mathbf{q}}_{T,j})\|_{L^2(\widehat D_T)}
.
\end{equation*}
By an application of the exponential decay argument from
\cite{MalqvistPeterseim2014,HenningPeterseim2013}
it can be shown that the right-hand side is controlled
by $H \, |T|^{1/2}$. The combination of the foregoing estimates
and the shape regularity yield the assertion.
\end{proof}

\subsection{Main result}
The central result of this paper is the a~priori quantification of the a~posteriori model estimator under the assumptions of stationarity and quantitative decorrelation. 
\begin{theorem}[a~priori error estimate for $\gamma$]
       \label{t:gammaestimate}
Let the diffusion tensor $\mathbf A$ satisfy assumptions
(A1)--(A3) from Subsection~\ref{ss:keyassumptions} and let $\ell\approx \lvert\log H\rvert$.
Then, $\gamma$ from Definition~\ref{d:modelest} satisfies
\begin{align*}
 \gamma
\lesssim
\lvert\log H\rvert^{4+d/2}
\left(H + \left(\frac{\varepsilon}{H}\right)^{d/2}\right).
\end{align*}
\end{theorem}
\noindent
Section~\ref{ss:proof} below is devoted to the proof of this theorem. The combination of 
Proposition~\ref{p:errorestimateQlocal}
and Theorem~\ref{t:gammaestimate} readily yields the desired a priori error bound of the numerical stochastic homogenization method of Section~\ref{s:notation}. 

\begin{corollary}[a~priori error estimate for the numerical method]
       \label{c:errorestimate}
Let the diffusion tensor $\mathbf A$ satisfy assumptions
(A1)--(A3) from Subsection~\ref{ss:keyassumptions}, let $\ell\approx \lvert\log H\rvert$, 
and let $\mathbf{u}$ solve \eqref{e:diff1drandweak2} and let
$u_H$ solve \eqref{e:uHdef} with right-hand side
$f\in L^2(D)$.
Then,
\begin{equation*}
\begin{aligned}
\sqrt{\mathbb{E}[ \| \mathbf{u}-u_H \|_{L^2(D)}^2] }
\lesssim 
\lvert\log H\rvert^{4+3d/2} 
 \left(H+\left(\frac{\varepsilon}{H}\right)^{d/2} \right)
  \|f\|_{L^2(D)}.
\end{aligned}
\end{equation*}
\end{corollary}
\noindent
Note that the logarithmic factor $|\log H|^{4+3d/2}$ in the preceding estimate is likely non-optimal; for instance, deriving and using sharper bounds on $|\nabla \mathbf{q}_{T,k}|$ that reflect the exponential decay of $\mathbf{q}_{T,k}$ outside of $T$ in the proof below would give rise to an improved estimate.

\subsection{Proof of the main result}\label{ss:proof}
This section is devoted to the proof of Theorem~
\ref{t:gammaestimate}.
Subsections~\ref{ss:modelerror} and 
\ref{ss:schur} provide the necessary variance bounds
for the entries of $\boldsymbol{\mathcal A}_H$.
The final Subsection~\ref{sss:proofMainThm} concludes
the proof of
Theorem~\ref{t:gammaestimate}.

\subsubsection{Variance bounds for the entries of 
                \eqref{e:AHdef}}\label{ss:modelerror}

\begin{theorem}\label{t:varestimate}
Let $\mathbf{A}$ be a random coefficient field subject to the assumptions (A1)--(A3). Then the entries of the upscaled operator $\boldsymbol{\mathcal A}_H$ defined in \eqref{e:AHdef} satisfy the variance estimate
\begin{align*}
\mathbb{E}\Big[\Big(\boldsymbol{\mathcal A}_H|_{T,K}-\bar{\mathcal{A}}_H|_{T,K}\Big)^2 \Big]
\lesssim
\frac{H^2}{|T|^2}+
 \frac{\ell^8}{|T|^2} \bigg(\frac{\varepsilon}{H}\bigg)^{d}.
\end{align*}
\end{theorem}
\begin{proof}
We apply Lemma~\ref{l:perturb} and assume that $D_T$
is a box and note that the error from this replacement is controlled
by $H$.
Our goal is to estimate the variance of $\boldsymbol{\mathcal A}_H|_{T,K}$ by means of the spectral gap inequality \eqref{SpectralGap}. To do so, we need to bound the Fr\'echet derivative of $\boldsymbol{\mathcal A}_H|_{T,K}$. By \eqref{e:AHdef} we have
\begin{align*}
\frac{\partial (\boldsymbol{\mathcal A}_H|_{T,K})_{jk}}{\partial \mathbf{A}} (\mathbf{\delta A})
=\frac{1}{|T|\,|K|}
\Bigg(&\delta_{T,K}\int_T (\mathbf{\delta A})_{jk}\,dx -
 e_j\cdot\int_K \mathbf{\delta A}\nabla \mathbf{q}_{T,k}\,dx 
\\&
-e_j\cdot\int_K \mathbf{A}\nabla \frac{\partial \mathbf{q}_{T,k}}{\partial \mathbf{A}}(\mathbf{\delta A})\,dx 
\Bigg).
\end{align*}
We define the auxiliary functions $\boldsymbol{\xi}_{T,K,j}\in W_{D_T}$ as the unique solution to the equation
\begin{equation}\label{e:auxpde}
\int_{D_T} \nabla w\cdot(\mathbf{A}^* \nabla \boldsymbol{\xi}_{T,K})\,dx
=
\int_{K} \nabla w \cdot(\mathbf{A}^* e_j)\,dx
\quad\text{for all }w\in W_{D_T}.
\end{equation}
Choosing $w=\frac{\partial \mathbf{q}_{T,k}}{\partial \mathbf{A}}(\mathbf{\delta A})$ as a test function, we may rewrite the Fr\'echet derivative of $\boldsymbol{\mathcal A}_H|_{T,K}$ as
\begin{multline}\label{e:char1}
\frac{\partial (\boldsymbol{\mathcal A}_H|_{T,K})_{jk}}{\partial \mathbf{A}} (\mathbf{\delta A})
=\frac{1}{|T|\,|K|}
\Bigg(\delta_{T,K}\int_T (\mathbf{\delta A})_{jk}\,dx -
 e_j\cdot\int_K \mathbf{\delta A}\nabla \mathbf{q}_{T,k}\,dx 
\\
-\int_{D_T} \nabla \boldsymbol{\xi}_{T,K,j} \cdot \mathbf{A}\nabla \frac{\partial \mathbf{q}_{T,k}}{\partial \mathbf{A}}(\mathbf{\delta A})\,dx 
\Bigg).
\end{multline}
The differentiation of \eqref{e:qTelementcorrELL} shows that
\begin{equation*}
\int_{D_T} \nabla w\cdot\bigg(\mathbf{A} \nabla \frac{\partial \mathbf{q}_{T,k}}{\partial \mathbf{A}}(\mathbf{\delta A})\bigg)\,dx
=
\int_{T} \nabla w \cdot(\delta\mathbf{A} e_k)\,dx
-\int_{D_T} \nabla w\cdot(\delta\mathbf{A} \nabla \mathbf{q}_{T,k})\,dx
\end{equation*}
for any $w\in W_{D_T}$. The particular choice $w=\boldsymbol{\xi}_{T,K,j}$ allows one to rewrite \eqref{e:char1} in the form
\begin{align*}
\frac{\partial (\boldsymbol{\mathcal A}_H|_{T,K})_{jk}}{\partial \mathbf{A}} (\mathbf{\delta A})
=
&
\frac{1}{|T|\,|K|}
\Bigg(\delta_{T,K}\int_T (\mathbf{\delta A})_{jk}\,dx -
 e_j\cdot\int_K \mathbf{\delta A}\nabla \mathbf{q}_{T,k}\,dx 
\\
&
\quad
-\int_{T} \nabla \boldsymbol{\xi}_{T,K,j} \cdot(\delta\mathbf{A} e_k)\,dx
+\int_{D_T} \nabla \boldsymbol{\xi}_{T,K,j} \cdot(\delta\mathbf{A} \nabla \mathbf{q}_{T,k})\,dx
\Bigg).
\end{align*}
This expression characterizes the 
($L^2$ representation of the)
Fr\'echet derivative of the entries of $\boldsymbol{\mathcal A}_H|_{T,K}$ with respect to 
the coefficient field $\mathbf{A}$ as
\begin{align*}
\frac{\partial (\boldsymbol{\mathcal A}_H|_{T,K})_{jk}}{\partial \mathbf{A}}
=\frac{1}{|T|\,|K|}
\Big(&\delta_{T,K} e_j\otimes e_k  \chi_T
- e_j\otimes \nabla \mathbf{q}_{T,k} \chi_K 
\\&
-\nabla \boldsymbol{\xi}_{T,K,j} \otimes e_k \chi_T
+\nabla \boldsymbol{\xi}_{T,K,j} \otimes \nabla \mathbf{q}_{T,k} \chi_{D_T}
\Big).
\end{align*}
Using the estimate \eqref{SpectralGap} of Assumption (A3), this readily yields
\begin{align*}
&\mathbb{E}\Big[\Big(\boldsymbol{\mathcal A}_H|_{T,K}-\bar{\mathcal{A}}_H|_{T,K}\Big)^2 \Big]
\\&
\lesssim
\frac{\varepsilon^d}{|T|^2 |K|^2}
\mathbb{E}\Bigg[
\int_\Rd \bigg|\fint_{B_\varepsilon(x)} \chi_T (\delta_{T,K}+|\nabla \boldsymbol{\xi}_{T,K}|) \,d\tilde x\bigg|^2 \,dx
\Bigg]
\\&~~~
+\frac{\varepsilon^d}{|T|^2 |K|^2}
\mathbb{E}\Bigg[\int_\Rd \bigg|\fint_{B_\varepsilon(x)} (\chi_K+\chi_{D_T} |\nabla \boldsymbol{\xi}_{T,K}|) |\nabla \mathbf{q}_{T}| \,d\tilde x\bigg|^2 \,dx
\Bigg].
\end{align*}
Jensen's inequality and H\"older's inequality then imply
\begin{align*}
&\mathbb{E}\Big[\Big(\boldsymbol{\mathcal A}_H|_{T,K}-\bar{\mathcal{A}}_H|_{T,K}\Big)^2 \Big]
\\&
\lesssim
\frac{\varepsilon^d}{|T|^2 |K|^2} \Bigg(|T|\delta_{T,K}
+ \mathbb{E}\Bigg[\int_{T} |\nabla \boldsymbol{\xi}_{T,K}|^2 \,dx \Bigg] + \mathbb{E}\Bigg[\int_{K} |\nabla \mathbf{q}_{T}|^2 \,dx\Bigg] \Bigg)
\\&~~~
+\frac{\varepsilon^d}{|T|^2 |K|^2}
\int_{D_T}
\mathbb{E}\Bigg[\bigg(\fint_{B_\varepsilon(x)} \chi_{D_T} |\nabla \boldsymbol{\xi}_{T,K}|^2 \,d\tilde x\bigg)^2 \Bigg]^{1/2}
\\&\qquad\qquad \qquad\qquad\quad \times
\mathbb{E}\Bigg[\bigg(\fint_{B_\varepsilon(x)} \chi_{D_T} |\nabla \mathbf{q}_{T}|^2 \,d\tilde x\bigg)^2 \Bigg]^{1/2} \,dx.
\end{align*}
At this point we make use of the assumption that $D_T$ is a box.
We first note that Lemma~\ref{l:schurcomplement} 
is still valid for the box $D_T$ (which need not necessarily match
with the triangulation $\T_H$)
because, after considering a larger patch
$\widehat{D}_T$ containing the box domain $D_T$ and 
applying Lemma~\ref{l:schurcomplement} there,
the restriction of resulting right-hand side $\hat b_{T,j}$
satisfies the properties from Lemma~\ref{l:schurcomplement}
for the box domain.
Using Lemma~\ref{l:schurcomplement} (which applies to 
$\boldsymbol{\xi}_{T,K}$, as its defining equation \eqref{e:auxpde} 
is of the same structure 
as \eqref{e:qTelementcorrELL}) and Lemma~\ref{l:RegularityAnnealed}
(which applies to the box $D_T$) below,
we deduce the estimates
\begin{align*}
\int_{D_T} \mathbb{E}\Bigg[\bigg(\fint_{B_\varepsilon(x)} |\nabla \mathbf{q}_{T}|^2 \,d\tilde x\bigg)^2 \Bigg] \,dx
&\leq C(\lambda,\Lambda,\rho) \ell^8 |T|,
\\
\int_{D_T} \mathbb{E}\Bigg[\bigg(\fint_{B_\varepsilon(x)} |\nabla \boldsymbol{\xi}_{T,K}|^2 \,d\tilde x\bigg)^2 \Bigg] \,dx
&\leq C(\lambda,\Lambda,\rho) \ell^8 |K|.
\end{align*}
Furthermore, a simple energy estimate yields
\begin{align*}
\mathbb{E}\bigg[\int_{D_T} |\nabla \mathbf{q}_{T}|^2  \,dx \bigg]
&\leq C(\lambda,\Lambda) |T|,
\\
\mathbb{E}\bigg[\int_{D_T} |\nabla \boldsymbol{\xi}_{T,K}|^2 \,dx \bigg] 
&\leq C(\lambda,\Lambda) |K|.
\end{align*}
Inserting these bounds as well as the relations $H^d \lesssim |T| \lesssim |K|\lesssim H^d$ in the previous estimate, we obtain the desired bound
\begin{align*}
&\mathbb{E}\Big[\Big(\boldsymbol{\mathcal A}_H|_{T,K}-\bar{\mathcal{A}}_H|_{T,K}\Big)^2 \Big]
\lesssim \frac{\varepsilon^d}{H^d |T|^2} \ell^8.
\end{align*}
\end{proof}

The following result is derived in the case of an equation on the full space $\mathbb{R}^d$ in \cite{DuerinckxOtto2019}, extending earlier results from \cite{ArmstrongDaniel,DuerinckxGloriaOtto2016}. Its proof in the case of the Dirichlet problem on a box is analogous, but requires a boundary regularity theory as derived in \cite{FischerRaithel}, as well as a regularity theory at edges and corners as an input;
we refer to the forthcoming work \cite{BellaFischerJosienRaithel}.
\begin{lemma}[Annealed large-scale $L^p$ theory for random elliptic operators on cubes]
\label{l:RegularityAnnealed}
Let $d\in \{2,3\}$ and let $\mathbf{A}$ be a random coefficient field subject to the assumptions (A1)--(A3). Let $Q\subset \mathbb{R}^d$ be a box, let $b\in L^2(Q)$, and let $\boldsymbol{u}\in L^2(\Omega;H^1_0(Q))$ be a solution to the linear elliptic PDE
\begin{align*}
-\nabla \cdot (\mathbf{A}\nabla \boldsymbol{u})&=\nabla \cdot b &&\text{on }Q,
\\
\boldsymbol{u}&\equiv 0 &&\text{on }\partial Q.
\end{align*}
Then for any $2\leq p<\infty$ and any $p<q<\infty$
a regularity estimate of the form
\begin{align*}
\fint_Q \mathbb{E}\Bigg[\bigg(\fint_{B_\varepsilon(x)} \chi_{Q} |\nabla \boldsymbol{u}|^2 \,d\tilde x\bigg)^{p/2} \Bigg] \,dx
\leq C(\lambda,\Lambda,\rho,p,q) \bigg(\fint_Q |b|^q \,dx\bigg)^{p/q}
\end{align*}
holds true.
\end{lemma}

\subsubsection{Schur complement representation of the element correctors}\label{ss:schur}

\begin{lemma}
\label{l:schurcomplement}
The element correctors $\mathbf{q}_{T,j}$ satisfy a PDE of the form
\begin{align*}
\nabla \cdot (\mathbf{A} \nabla \mathbf{q}_{T,j})=\nabla \cdot (\mathbf{A}  e_j\chi_T + b_{T,j}) \qquad\text{on }D_T
\end{align*}
for some $b_{T,j}$ with
\begin{align*}
\int_{D_T} |b_{T,j}|^{9/2} \,dx \lesssim \ell^{8} |T|.
\end{align*}
\end{lemma}
\begin{proof}
The proof proceeds by rewriting the defining equation of the element correctors \eqref{e:qTelementcorrELL} as a Schur complement problem.

Let $f_{T,j}:=-\nabla \cdot (\mathbf{A} e_j \chi_{T})\in H^{-1}(D_T)$.
 Let $\mathcal{L}:H^1_0(D_T)\rightarrow H^{-1}(D_T)$ be defined as 
$\mathcal{L}u:=-\nabla \cdot (\mathbf{A} \nabla u)$ 
(where the operator $\mathcal{L}$ is not to be confused with the
 upscaled operator $\boldsymbol{\mathcal A}_H$), and let
 $\mathcal{I}_{H,D_T}: H^1_0(D_T) \rightarrow \hat V_{H}(D_T)$ be
 the concatenation of extension by zero,  quasi-interpolation $I_{H}$,
 and restriction to the patch $D_T$. 
Here, by $\hat V_H(D_T)$ we denote the 
range $\mathcal{I}_{H,D_T}(H^1_0(D_T))$ of the 
operator $\mathcal{I}_{H,D,T}$. Note that 
$\hat V_H(D_T)$ is a subspace of the space of $P_1$ finite element functions on the patch $D_T$ with \emph{arbitrary boundary values},
but with zero boundary values on $\partial D\cap D_T$.
Denote by $p_{T,j}\in \hat V_{H}'(D_T)$ the Lagrange multiplier associated with the constraint $\mathcal{I}_{H,D_T} \mathbf{q}_{T,j}=0$ (note that this constraint is equivalent to $I_H \mathbf{q}_{T,j}=0$).

The element correctors $\mathbf{q}_{T,j}\in H^1_0(D_T)$ are then determined by the Schur complement problem
\begin{align}
\label{e:SchurComplement}
\begin{pmatrix}
\mathcal{L}&\mathcal{I}_{H,D_T}^t
\\
\mathcal{I}_{H,D_T}&0
\end{pmatrix}
\begin{pmatrix}
\mathbf{q}_{T,j}
\\
p_{T,j}
\end{pmatrix}
=
\begin{pmatrix}
f_{T,j}
\\
0
\end{pmatrix}.
\end{align}
By the standard theory for Schur complement problems, we have
\begin{align}
\label{e:defp}
p_{T,j} = (\mathcal{I}_{H,D_T} \mathcal{L}^{-1} \mathcal{I}_{H,D_T}^t)^{-1} \mathcal{I}_{H,D_T} \mathcal{L}^{-1} f_{T,j}.
\end{align}

By the Lax-Milgram theorem and the uniform ellipticity and boundedness of $\mathbf{A}$ the operator $\mathcal{L}:H^1_0(D_T)\rightarrow H^{-1}(D_T)$ is invertible and the operator norm of its inverse is bounded by a constant. 
Moreover, we have $\|\mathcal{L}^{-1}v\|_{H^1_0(D_T)}\approx \|v\|_{H^{-1}(D_T)}$. The latter is seen directly by $\mathcal{L}\mathcal{L}^{-1}v=v$ for all $v\in H^{-1}(D_T)$ and $\|\mathcal{L}w\|_{H^{-1}(D_T)}\lesssim \|w\|_{H^1_0(D_T)}$ for all $w\in H^1_0(D_T)$.

Set $\hat f:=\mathcal{L}^{-1} f_{T,j}$. We then have by the Poincar\'e inequality on the patch $D_T$, the bound on the operator norm of $\mathcal{L}^{-1}$, and the definition of $f_{T,j}$
\begin{align}
\|\hat f\|_{L^2(D_T)}&\lesssim \ell H \|\hat f\|_{H^1_0(D_T)}\lesssim \ell H \|f_{T,j}\|_{H^{-1}(D_T)} \leq \ell H \|\mathbf{A} e_j \chi_T\|_{L^2(D_T)}
\nonumber
\\&
\lesssim \ell H |T|^{1/2}.
\label{e:EstimateL2f}
\end{align}

Reformulating \eqref{e:defp}, $p_{T,j}$ is given by the solution to the equation
\begin{align}
\label{e:pweakform}
\langle \mathcal{L}^{-1} \mathcal{I}_{H,D_T}^t p_{T,j} , \mathcal{I}_{H,D_T}^t w \rangle
= \langle \hat f, \mathcal{I}_{H,D_T}^t w \rangle
\quad\quad\quad\quad\text{for all }w\in \hat V_H'(D_T). 
\end{align}

The quadratic form associated with the operator $\mathcal{I}_{H,D_T} \mathcal{L}^{-1} \mathcal{I}_{H,D_T}^t:\hat V_{H}'(D_T)\rightarrow \hat V_{H}(D_T)$ is coercive. This follows from the energy estimate
\begin{equation}
\label{e:coer}
\begin{aligned}
\langle \mathcal{L}^{-1} \mathcal{I}_{H,D_T}^t v , \mathcal{I}_{H,D_T}^t v \rangle
&
=\langle \mathcal{L}^{-1} \mathcal{I}_{H,D_T}^t v , \mathcal{L} \mathcal{L}^{-1} \mathcal{I}_{H,D_T}^t v \rangle
\\&
= \int_{D_T} \mathbf{A}\nabla (\mathcal{L}^{-1}\mathcal{I}_{H,D_T}^t v) \cdot \nabla (\mathcal{L}^{-1}\mathcal{I}_{H,D_T}^t v) \,dx
\\&
\gtrsim \int_{D_T} |\nabla (\mathcal{L}^{-1}\mathcal{I}_{H,D_T}^t v)|^2 \,dx
\\&
=\|\mathcal{L}^{-1}\mathcal{I}_{H,D_T}^t v\|_{H^1_0(D_T)}^2
\gtrsim \|\mathcal{I}_{H,D_T}^t v\|_{H^{-1}(D_T)}^2
\end{aligned}
\end{equation}
where in the first estimate the lower bound from (A1) has been used.
Thus, the Lax-Milgram theorem yields the existence of a unique solution $p_{T,j}$ to the problem \eqref{e:pweakform}. 
We furthermore note that Remark~\ref{r:submesh} 
(with $\tilde V_H$ defined with respect to two barycentric
refinements) imply
\begin{align*}
\|\mathcal{I}_{H,D_T}^t p_{T,j}\|_{L^2(D_T)}&= \sup_{0\neq v\in H^1_0(D_T)} \frac{\langle\mathcal{I}_{H,D_T}^t p_{T,j} ,v \rangle}{\|v\|_{L^2(D_T)}}
= \sup_{0\neq v\in H^1_0(D_T)} \frac{\langle p_{T,j} ,\mathcal{I}_{H,D_T}v\rangle}{\|v\|_{L^2(D_T)}}
\\&
\lesssim \sup_{0\neq \tilde v\in \tilde V_{H}(D_T)} \frac{\langle \mathcal{I}_{H,D_T}^t p_{T,j} ,\tilde v \rangle}{\|\tilde v\|_{L^2(D_T)}}
.
\end{align*}
A standard inverse estimate for finite element functions
on the submesh therefore yields 
\begin{align*}
\|\mathcal{I}_{H,D_T}^t p_{T,j}\|_{L^2(D_T)}
&
\lesssim
\sup_{0\neq v_H\in \tilde V(D_T)} \frac{\langle \mathcal{I}_{H,D_T}^t p_{T,j} ,\tilde v \rangle}{H\|\tilde v\|_{H^1_0(D_T)}}
\\
&\lesssim \sup_{0\neq v \in H^1_0(D_T)} \frac{\langle \mathcal{I}_{H,D_T}^t p_{T,j} ,v \rangle}{H\|v\|_{H^1_0(D_T)}}=H^{-1} \|\mathcal{I}_{H,D_T}^t p_{T,j}\|_{H^{-1}(D_T)}.
\end{align*}
In combination with \eqref{e:coer}, \eqref{e:pweakform} and \eqref{e:EstimateL2f} this implies 
\begin{align}
\label{e:Boundp}
\|\mathcal{I}_{H,D_T}^t p_{T,j}\|_{L^2(D_T)}\lesssim \ell H^{-1}|T|^{1/2}.
\end{align}

In total, from \eqref{e:SchurComplement} we see that the element correctors $\mathbf{q}_{T,j}$ solve an equation of the form
\begin{align*}
-\nabla \cdot (\mathbf{A} \nabla \mathbf{q}_{T,j})=-\nabla \cdot (\mathbf{A} e_j \chi_T) + \mathcal{I}_{H,D_T}^t p_{T,j}.
\end{align*}
We now claim that this may be rewritten as
\begin{align*}
-\nabla \cdot (\mathbf{A} \nabla \mathbf{q}_{T,j})=-\nabla \cdot (\mathbf{A}  e_j\chi_T + b_{T,j})
\end{align*}
for some $b_{T,j}$ with
\begin{align*}
\int_{D_T} |b_{T,j}|^{9/2} \,dx \lesssim \ell^8 |T|.
\end{align*}
To see this, one may for example choose $b_{T,j}:=\nabla v$ for $v$ solving $-\Delta v = \mathcal{I}_{H,D_T}^t p_{T,j}$ with homogeneous Dirichlet boundary conditions on a ball $B_{C\ell H}(y)$ which contains $D_T$ (where we extend $\mathcal{I}_{H,D_T}^t p_{T,j}$ to $B_{C\ell H}(y)$ by zero outside of $Q$).
Using elliptic regularity theory and \eqref{e:Boundp} one then has
\begin{align*}
\int_{D_T} |D^2 v|^2 \,dx \lesssim \|\mathcal{I}_{H,D_T}^t p_{T,j}\|_{L^2(D_T)}^2\lesssim \ell^2 H^{-2} |T|
\end{align*}
as well as $\int_{D_T} |\nabla v|^2 \,dx \lesssim \ell^2 H^2\|\mathcal{I}_{H,D_T}^t p_{T,j}\|_{L^2(D_T)}^2 \lesssim \ell^4 |T|$. Finally, the Sobolev embedding and a scaling argument imply $\int_{D_T} |b_{T,j}|^{9/2} \,dx\lesssim \ell^{8} |T|$.
\end{proof}

\subsubsection{The proof of Theorem~\ref{t:gammaestimate}}\label{sss:proofMainThm}
\begin{proof}[Proof of Theorem~\ref{t:gammaestimate}]
 Denote for a given $T\in\T_H$, the index set
$J:=\{K\in\T_H : K\cap\nei^\ell(T)\neq\emptyset\}$
and, for $K\in J$,
$\boldsymbol{v}_K:=|T|(\boldsymbol{\mathcal A}_H|_{T,K}-\bar{ \mathcal A}_H|_{T,K})$.
Then,
$$
\boldsymbol{X}(T)=\max_{K\in J} |\boldsymbol{v}_K|
$$
and elementary arguments show that
\begin{align*}
\mathbb E[\boldsymbol{X}(T)^2]
\leq
\mathbb E\bigg[\sum_{K\in J} |\boldsymbol{v}_K|^2\bigg]
\leq
\card(J)\max_{K\in J} \mathbb E[|\boldsymbol{v}_K|^2].
\end{align*}
Theorem~\ref{t:varestimate} shows that 
$$
\mathbb E[|\boldsymbol{v}_K|^2]
\lesssim 
H^2 + \ell^8 \bigg(\frac{\varepsilon}{H}\bigg)^{d}.
$$
This and $\card(J)\lesssim \ell^d$ by quasi-uniformity of the mesh imply
\begin{align*}
\max_{T\in\T_H} \Big( \sqrt{\mathbb{E}[\boldsymbol{X}(T)^2]}\Big)
\lesssim
\sqrt{\card(J)}
\left(H +\ell^4 \left(\frac{\varepsilon}{H}\right)^{d/2}\right)
\lesssim
\ell^{4+d/2} 
\left(H + \left(\frac{\varepsilon}{H}\right)^{d/2}\right).
\end{align*}
\end{proof}

\section*{Acknowledgments}
This work was initiated while the authors enjoyed
the kind hospitality of the Hausdorff Institute for Mathematics 
in Bonn during the trimester program \emph{Multiscale Problems: Algorithms, Numerical Analysis and Computation}. D. Peterseim would like to acknowledge the kind hospitality of the Erwin Schr\"odinger International Institute for Mathematics and Physics (ESI) where parts of this research were developed under the frame of the Thematic Programme \emph{Numerical Analysis of Complex PDE Models in the Sciences}. D.~Peterseim acknowledges funding from the European Research Council (ERC) under the European Union's Horizon 2020 research and innovation programme (Grant agreement No. 865751).

{\footnotesize
\bibliographystyle{plain}
\bibliography{homcoeff}
}

\end{document}